\documentclass{article}

\usepackage{amsmath,mathtools,amssymb,amsfonts,amsthm,graphicx,color,verbatim,enumerate,dsfont,pslatex,fourier}

\usepackage[numbers]{natbib}

\definecolor{Blue}{rgb}{0.3,0.3,0.9}
\usepackage[margin=1.707in]{geometry}

\definecolor{e-mail}{rgb}{0,.40,.80}
\definecolor{reference}{rgb}{.20,.60,.22}
\definecolor{citation}{rgb}{0,.40,.80}
\usepackage[colorlinks, linkcolor=reference,
            citecolor=citation,
           urlcolor=e-mail]{hyperref}

\newtheorem{theorem}{Theorem}[section]

\newtheorem{proposition}[theorem]{Proposition}

\numberwithin{equation}{section}

\theoremstyle{remark}

\newcommand{\classification}[2]{{\small\noindent 2010 \textit{MSC:}\ \  #1 (primary); \   #2 (secondary)\vspace{0pc}}}

\begin{document}

\title{A Multi-Class Extension of the Mean Field Bolker-Pacala Population Model\footnote{This work was supported in part 
by NSF grant DMS-1515800 and 
by PSC-CUNY grant \#68170-00 46 and by funds provided by UNC Charlotte.}\protect}
\author{Mariya Bessonov
 \\\small Department of Mathematics\\[-0.8ex]
\small CUNY New York City College of Technology, NY 11201 \\
\\
Stanislav Molchanov, Joseph Whitmeyer
 \\\small Department of Mathematics and Statistics\\[-0.8ex]
\small University of North Carolina at Charlotte, Charlotte, NC 28223\\}

\date{}
\maketitle

\begin{abstract}
We extend our earlier mean field approximation of the Bolker-Pacala model of population dynamics by dividing the population into $N$ classes, using a mean field approximation for each class but also allowing migration between classes as well as possibly suppressive influence of the population of one class over another class.  For $N \geq 2$, we obtain one symmetric non-trivial equilibrium for the system and give global limit theorems.  For $N=2$, we calculate all equilibrium solutions, which, under additional conditions, include multiple non-trivial equilibria.  Lastly, we prove geometric ergodicity regardless of the number of classes when there is no population suppression across the classes.
\\

\classification{92D25}{60J10}
\section{Introduction}

\end{abstract}

The Bolker-Pacala (BP) model of population dynamics, from biology, involves processes of birth, death, 
and migration, as well as competition or suppression.  In a previous paper \cite{BMW14}, we analyzed a 
mean-field approximation of the BP model, obtaining results such as local and global central limit theorems 
for population size.  While that model treated basic population questions, in this paper we extend the 
mean-field approach to address additional topics.

Specifically, we consider a population now divided into $N$ classes or ``boxes," and analyze a mean-field 
approximation for each box.  We allow the possibility of migration between boxes and of competitive effects or 
the suppression of the population in one box by the population in other boxes.  While it is possible to think 
of the boxes as geographical areas, it is perhaps most intriguing to view them as segments of a population such 
as social classes.  In this case, the $N$-box BP model becomes a model of social stratification.  Migration 
between boxes corresponds, then, to social mobility with the parameters for migration giving the rates of social 
mobiliy.  The parameters for competition \emph{within} boxes may correspond to constraints, such as economic 
constraints, on the size of classes.  It is questionable whether suppression \emph{across} classes would exist 
or whether these parameters would be 0.

For $N = 2$ and 3 we obtain two new results:
\begin{itemize}
\item  first, allowing suppression of population across boxes 
creates the possibility of more than one non-trivial equilibrium population level; 
\item second, when there is only one 
non-trivial equilibrium, such as in the absence of such cross-box suppression, the equilibrium level is not 
affected by migration from one box to another. 
\end{itemize}

The paper is laid out as follows.  In Section~\ref{sec:description}, we describe the $N$-box mean field Bolker-Pacala model.  
In the following Sections~\ref{sec:Nbox} and~\ref{sec:GlobalN}, we give a global analysis, showing the existence of one symmetric, non-trivial equilibrium point, and presenting global limit theorems for $N \geq 2$.  Exact results for $N = 2$ are given there.   In Section~\ref{sec:Ergodicity}, we establish the geometric ergodicity of the 
process regardless of the number of boxes when population suppression from other boxes is 0, and gives the 
equilibrium point when internal competition is identical for all boxes.


\section{Preliminaries: description of the process}\label{sec:description}
We begin with an introduction of the general Bolker-Pacala model, which can be formulated as follows. 
There is some initial homogeneous population on $\mathbb{R}^d$, that is, a locally 
finite point process $$n_0(\Gamma) = \#\text{(particles in $\Gamma$ at time $t=0$)},$$ 
where $\Gamma$ denotes a bounded and connected region in $\mathbb{R}^d$. We refer to individual 
members of the population as particles and the location of a particle on  $\mathbb{R}^d$ as the 
site of that particle. For instance, one can consider $n_0(\Gamma)$ to be a Poissonian point 
field with intensity $\rho > 0$, 
i.e., $$P\{n_0(S) = k\} = \exp(-\rho |S|)\frac{(\rho |S|)^k}{k!},\ k=0,1,2,\ldots$$ 
where $S\subset\Gamma$ and $|S|$ represents the (finite) Lebesgue measure of $S$, and the number 
of points in each set of any disjoint collection of subsets of $\Gamma$ is independent.  
The following rules dictate the evolution of the field:
\begin{enumerate}
\item[i)] Each particle, independent of the others, during time interval $(t,t+dt)$ can 
produce a new particle (offspring or seed) with probability $\beta \, dt + o(dt^2) = A^+ dt + o(dt^2)$, $A^+>0$.  The 
initial particle remains at its initial position $x$ but the offspring jumps to $x+z+dz$ 
with probability $$a^+(z) dz,\quad A^+=\int\limits_{\mathbb{R}^d} \!\! a^+(x) dx.$$  Note 
that this can be seen equivalently as two random events, the birth of a particle and its 
dispersal, as in Bolker and Pacala's presentation \cite{bp99,bpn03}, or as a single random event, 
as in our model.  (We stress that this differs from the classical branching process, in which the 
``parental" particle and its offspring commence independent motion from the same point.)  We will 
assume that all offspring evolve independently according to the same rules. 
\item[ii)] Each particle at point $x$ during the time interval $(t,t+{d}t)$ dies with 
probability $\mu\, {d}t + o(dt^2)$, where  $\mu$ is the mortality rate.
\item[iii)] The competition factor leads to many interesting properties in this model.  If two particles are located at the 
points $x,y \in \mathbb{R}^d$, then each of them dies with probability $a^-(x-y)dt + o(dt^2)$ during 
the time interval $(t,t+dt)$ (due to independence, the probability that both die is $o(dt^2)$). This requires, of course, 
that $a^-(\cdot)$ be integrable; set $$A^- = \int\limits_{\mathbb{R}^d} \!\! a^-(z) {d}z.$$ 
The total effect of competition on a particle is the sum of the effects of competition with all individual particles. 
\end{enumerate}
Here we have interacting particles, in contrast to the usual branching process.  One can expect 
physically that for arbitrary non-trivial competition ($a^-\in C(\mathbb{R}^d)$, $A^- > 0$), there 
will exist a limiting distribution of the particles.  At each site $x\in\mathbb{R}^d$, with population at time $t$ 
given by $n(t,x)$, three rates are relevant, the birth rate $\beta$ and mortality rate $\mu$, each 
proportional to $n(t,x)$ and the death rate due to competition, proportional to $n(t,x)^2$.   
Heuristically, when $n(t,x)$ is small the linear effects will dominate. Thus, if $\beta > \mu$ 
the population is expected to increase.  As the population grows and $n(t,x)$ becomes large enough, 
however, the quadratic effect due to competition will become 
increasingly dominant, which will prevent unlimited population growth.  At present, this fact has been proven 
only under strong restrictions on $a^+$ and $a^-$ \cite{fkkk2}.

\section{The $N$-box model}\label{sec:Nbox}
In the first part of Section~\ref{1box}, we recall the mean-field approximation to the 
Bolker Pacala model from \cite{BMW14}, in which we considered the $1$-box model.
In Section~\ref{Nbox}, we generalize our mean-field approximation to the $N$-box model.

\subsection{The $1$-box model}\label{1box}
The mean field approximation, ``$1$-box model'' of the BP process from \cite{BMW14}
led to the special Markov chain: the logistic random walk on the half-axis 
$\mathbb{Z}_{+} = \{0,1,2,\ldots \}$. In this model, we considered a system of particles (thinking 
of particles as individual members of some population).
All particles live on the lattice, $\mathbb{Z}^d$. Each lattice point 
$\mathbf{x}$ has an associated square $\mathbf{x}+[0,1)^d$, 
and the number of particles at $\mathbf{x}$ represents 
the number of inhabitants in the continuous model of 
that square in $\mathbb{R}^d$ that is associated with $\mathbf{x}\in\mathbb{Z}^d$.  

We let $Q_{L} \subset \mathbb{Z}^d$ be a box with $|Q_{L}| = L$, $L$ a large 
parameter, and suppose that no particles exist outside of $Q_{L}$.

We modify the notation from \cite{BMW14} slightly to match the notation 
in this paper. We recall the migration rate between sites on the lattice and competition rate, 
at which a particle at $\mathbf{x}$ outcompetes another particle at $\mathbf{y}$, in the $1$-box model: 
\begin{align*}
 a^+(\mathbf{x},\mathbf{y}) \equiv \frac{a^+}{L}\quad & \text{ for } \mathbf{x}, \mathbf{y}\in Q_L\cap\mathbb{Z}^d, \\
 a^-(\mathbf{x},\mathbf{y}) \equiv \frac{a^-}{L^2}\quad & \text{ for } \mathbf{x}, \mathbf{y}\in Q_L\cap\mathbb{Z}^d
\end{align*}
for constants $a^+, a^- \ge 0$. With such rates, the distribution of a particle after 
a jump due to migration is uniform on $Q_L$. Let  $\beta$ and $\mu$
be the birth and mortality rates, respectively. We assume that $\beta > \mu$.

If $n(t,\mathbf{x})$ represents the number of particles at site $\mathbf{x}\in Q_L\cap\mathbb{Z}^d$
(we do not restrict the number of particles per site), then
\begin{equation*}
 N_L(t) = \sum_{\mathbf{x}\in Q_L\cap\mathbb{Z}^d} n(t,\mathbf{x})
\end{equation*}
is the total number of particles in $Q_L$ at time $t$. $N_L(t)$ is a Markov process,
which we call the ``logistic'' Markov chain.

The transition rates for $N_L(t)$ are 
\begin{displaymath}
  P\left(N_L(t+dt) = j \:|\: N_L(t) = n  \right) =  \left\{
     \begin{array}{ll}
       n\beta\,dt + o(dt^2)& \text{ if } j = n+1\\
       n\mu\,dt + a^-\cdot n^2/L \,dt + o(dt^2)& \text{ if } j = n-1\\
       o(dt^2)& \text{ otherwise }
     \end{array}
   \right.
\end{displaymath} 
We observe that if $N_L(t)$ is large, the random walk has a left drift,
whereas if $N_L(t)$ is small, the random walk has a drift to the right. 
An important point is the equilibrium point, $n_L^*$, where the rates
to the left and to the right are equal, that is,
\begin{equation*}
 \beta n_L^* = \mu n_L^* + \frac{a^-\cdot n_L^{*2}}{L},
\end{equation*}
Thus,
\begin{equation*}
 n_L^* = \left\lfloor \frac{L(\beta-\mu)}{a^-} \right\rfloor.
\end{equation*}
We showed in \cite{BMW14} that as $L\to\infty$, $N_L(t)$ 
tends quickly to a neighborhood of $n_L^*$ and afterward fluctuates randomly around $n_L^*$.
See \cite{BMW14} for further results including a local Central Limit Theorem and large deviations.

\subsection{The $N$-box model}\label{Nbox}
The more general $N$-box model gives rise to a random walk on 
$$(\mathbb{Z}_{+})^{N} = \{ (n_1,n_2,\ldots, n_N) \:|\: n_i \in \mathbb{Z}_{+}, 1\leq i\leq N \}.$$
Consider a system of $N$ disjoint rectangles $Q_{i,L}\subset\mathbb{R}^2$,
$i=1,2,\ldots,N$, with $$\left|Q_{i,L}\cap\mathbb{Z}^2\right| = L.$$
 As in the usual BP model, 
 introduce the migration potential $a^{+}$ and the competition potential $a^{-}$
 that are constant on each $Q_{i,L}$. For $\mathbf{x}\in Q_{i,L}, \mathbf{y}\in Q_{j,L}$,
\begin{equation}
 a^{-}_{L}(\mathbf{x},\mathbf{y}) = a^{-}_{ij}/L^{2}, \qquad\qquad\qquad i,j = 1,2,\ldots,N,
\end{equation}
 and
\begin{equation}
 a^{+}_{L}(\mathbf{x},\mathbf{y}) = a^{+}_{ij}/L, \qquad\qquad\qquad i,j = 1,2,\ldots,N.
\end{equation}
Specifically, $a^{-}_{ij}$ indicates the depressive effect on the population in 
box $i$ due to the population in box $j$ (i.e., competition between boxes $i$ and $j$),
while $a^{+}_{L}(\mathbf{x},\mathbf{y})$ is the rate of migration
from  $\mathbf{x}\in Q_{i,L}$ to $\mathbf{y}\in Q_{j,L}$.

Let  $\bigcup_{i=1}^{N} Q_{i,L} = Q_L$. Then set
\begin{equation*}
A^{+}_{i} := \sum_{\mathbf{y}\in Q_L} a^{+}(\mathbf{x},\mathbf{y}) = \sum_{j=1}^{N} a^{+}_{ij},\quad
 A^{-}_{i} := \sum_{\mathbf{y}\in Q_L} a^{-}(\mathbf{x},\mathbf{y}) = \sum_{j=1}^{N} a^{-}_{ij}
\end{equation*}
Assume that $$A_i^+, A_i^- \leq A < \infty$$ uniformly in $L$.
In this setup, the number of squares $N$ is fixed. 
The parameters $\beta_i,\mu_i > 0$ represent the natural (biological) birth and death 
rates of particles in box $i$, $i=1,\ldots,N$, respectively. 

The population in each square $Q_{i,L}$, $i=1,\ldots,N$, at time $t$ will be represented by 
\begin{align}\label{RW}
\boldsymbol{n}(t) = \{n_1(t),n_2(t),\ldots,n_N(t)\},
\end{align}a continuous time random 
walk on $(\mathbb{Z}_{+})^{N}$ with rates obtained from, for $i,j = 1,2,\ldots,N$,
\begin{align}
 \boldsymbol{n}(&t + dt | \boldsymbol{n}(t)) \\ \notag&= \boldsymbol{n}(t) + 
\begin{dcases}
 \, e_i & \text{w. pr. }  \beta_i n_i(t)dt + o(dt^2)\\
-e_i & \text{w. pr. } \mu_i n_i(t)dt  +\frac{n_i(t)}{L}\sum_{j=1}^{N}a^{-}_{ij}n_j(t)dt + o(dt^2)\\
e_{j} - e_{i} & \text{w. pr. } n_i(t)a^{+}_{ij}dt + o(dt^2), \quad j\neq i \\  
 0 & \text{w. pr. } 1- \sum_{i=1}^{N}(\beta_i+\mu_i)n_i(t)dt \\
 & \qquad \qquad - \frac{1}{L}\sum_{i,j}n_i(t)n_j(t)a^-_{ij}dt +\sum_{i,j}n_i(t)a^+_{ij} + o(dt^2)\\
 \text{other } & \text{w. pr. } o(dt^2)
\end{dcases}
\end{align}
where $e_i$ is the vector with $1$ in the $i^{th}$ position and $0$ everywhere else.

We define the \emph{transition function} $p\left(\mathbf{n}(t),\mathbf{n}(t) + \mathbf{k}\right)$ 
from the principal probabilities above, that is,
\begin{align}\label{tf}
 p(\mathbf{n}(t),&\mathbf{n}(t) + \mathbf{k}) \\&=
\begin{dcases}
\beta_i n_i(t) &  \, \mathbf{k} = e_i \\
\mu_i n_i(t) +\frac{n_i(t)}{L}\sum_{j=1}^{N}a^{-}_{ij}n_j(t) & \, \mathbf{k} = -e_i \\
n_i(t)a^{+}_{ij} & \, \mathbf{k} = e_{j} - e_{i},\,  j\neq i \\  
- \sum_{i=1}^{N}(\beta_i+\mu_i)n_i(t) - \frac{1}{L}\sum_{i,j}n_i(t)n_j(t)a^-_{ij} + & \, \mathbf{k} = 0 \\
\qquad +\sum_{i,j}n_i(t)a^+_{ij} & \\
0 & \, \text{all other } \mathbf{k}
\end{dcases}\notag
\end{align}

\section{Global analysis for $N$ boxes}\label{sec:GlobalN}

\subsection{Preliminaries}
Let us temporarily fix $L$.  
We set $$\frac{n_i(t)}{L} := z_i(t), \qquad i=1,\ldots,N.$$
Define 
\begin{equation*}
 f_L(\mathbf{z}(t),\mathbf{k}) := \frac{1}{L} p(\mathbf{n}(t),\mathbf{n}(t) + \mathbf{k}),
\end{equation*} 
where $\mathbf{z}(t) = (z_1(t),\ldots,z_N(t))$, $\mathbf{n}(t) = (n_1(t),\ldots,n_N(t))$, and
$\mathbf{k} = (k_1,\ldots,k_N)$, $k_i = 1, 0, \text{ or } -1$ for $i=1,\ldots,N$,
and $p$ is the transition function (\ref{tf}).
Then
        \begin{align*}
         f_{L}(\mathbf{z}(t),&\mathbf{k})= \left\{
        \begin{array}{lll}
        \beta_i z_i & \mathbf{k}= e_i, & i=1,\ldots,N\\
        \mu_i z_i + a_{i,i}^- z_i^2 + \displaystyle\sum_{j\ne i}a_{i,j}^- z_iz_j & \mathbf{k}= -e_i, & i=1,\ldots,N\\
        a^{+}_{i,j} z_i & \mathbf{k} = e_j - e_i, & i,j=1,\ldots,N; i\neq j \\
       - \sum_{i=1}^{N}(\beta_i+\mu_i)z_i(t)   & \, \mathbf{k} = 0 \\
       \qquad - L\sum_{i,j}z_i(t)z_j(t)a^-_{ij}&\\
\qquad +\sum_{i,j}z_i(t)a^+_{ij} & \\        0 & \text{ otherwise}&\\
        \end{array}
        \right.
        \end{align*}  
Note that $f_L(\mathbf{z}(t),\mathbf{k})$ does not, in fact, depend on $L$.

Set the migration rate out of box $i$ $$M_i^+ := \sum_{j\neq i}a^+_{i,j}.$$
For the functional limit theorems to follow, define
for $i=1,\ldots, N$,
	\begin{align}
   F_i(\mathbf{z}(t)) :&= \displaystyle\sum_{k_i = -1}^{1} k_if(\mathbf{z}(t),\cdot)\label{dfs1}
   \left(\beta_i - \mu_i -M^+_i\right)z_i - a^{-}_{i,i} z_i^2 - 
   \displaystyle\sum_{j\ne i}a_{i,j}^- z_iz_j + \sum_{j\neq i}a^{+}_{j,i} z_j 
	\end{align}
and consider the system of differential equations
	\begin{equation}\label{dfs2}
	\frac{d \mathbf{z}(t)}{dt} = \mathbf{F}(\mathbf{z}(t))
	\end{equation}
An equilibrium for the system occurs precisely at the points where
\begin{equation}\label{dfs3}
	\mathbf{0} = \mathbf{F}(\mathbf{z}),
	\end{equation}
with one solution being $\mathbf{z} \equiv \mathbf{0}$. 

Set $p_i := \beta_i - \mu_i - M^+_i$. In matrix form, we have the equation
$$
A\left[\begin{array}{c}
z_1\\
\vdots\\
z_N
\end{array}\right]
+ B \left[\begin{array}{c}
z^2_1\\
\vdots\\
z^2_N
\end{array}\right] = \mathbf{0},
$$
where $B$ is a diagonal matrix:
\[A=\left[\begin{array}{ccccc}
p_{1} & a^+_{2,1} & a^+_{3,1} & \hdots & a^+_{N,1} \\
a^+_{1,2} & p_{2} & a^+_{3,2} & \hdots & a^+_{N,2} \\
\vdots&& \ddots && \vdots\\
\vdots&&& \ddots &\vdots\\
a^+_{1,N}&\hdots&&a^+_{N-1,N}& p_{N}
\end{array}\right],\ \
B=\left[\begin{array}{cccc}
a^-_{1,1} &  &  &    \\
 & a_{2,2}^- && \makebox(0,0){\text{\huge0}} \\
\makebox(0,0){\text{\huge0}}&& \ddots & \\
&&& a^-_{N,N}
\end{array}\right].\]

When $a^+_{i,j} = 0$ and $a^-_{i,j} = 0, i\neq j$, that is, 
there is no migration between boxes and no suppression across boxes,
there is a unique non-zero equilibrium $$z_i = \frac{\beta_i - \mu_i}{a^-_{i,i}},\ \ i=1,...,N.$$
This is, as would be expected, essentially the equilibrium for $N$ 
distinct, independent ``single box" mean field Bolker Pacala models, as found in \cite{BMW14}.


\subsection{More on equilibrium points}
We assume, in this section, symmetric conditions, that is, that conditions are identical for all boxes.  Thus, 
the biological birth and mortality rates are the same in each box:
$$\beta_i \equiv \beta \text{ and } \mu_i \equiv \mu, \,\, i=1,2,\ldots.$$ The ``inner'' competition rates within boxes are equal, satisfying 
$$a^-_{I} := a^{-}_{ii}, \qquad i=1,2,\ldots$$ and ``outer'' competition (from box to box) is the same 
$$a^-_{O} := a^{-}_{ij}, \qquad i\neq j.$$ 
We also set the common migration rate $$a^+:= a^{+}_{ij}, \qquad i\neq j.$$
So that the system does not inevitably die out, we assume that $\beta >\mu$.  

We begin with the case of two boxes ($N = 2$) or classes. The system \eqref{dfs2} may have up to four distinct 
non-negative singular points, that is, solutions of~\eqref{dfs3}.  All four solutions are real and non-negative only if 
\begin{equation}\label{N2cond}
a^{-}_{O}>a^{-}_{I}\ \ \text{ and }\ \  \beta -\mu > 2a^{+} \frac{a^{-}_{O}+a^{-}_{I}}{a^{-}_{O}-a^{-}_{I}} 
\end{equation}
They are as follows:  
\begin{enumerate}[1)]
	\item The trivial singular point, an unstable equilibrium for $\beta > \mu$, at $(0,0)$.

	\item $\left(\frac{\beta-\mu}{a^{-}_{I}+a^{-}_{O}},\frac{\beta-\mu}{a^{-}_{I}+a^{-}_{O}}\right)$, which always exists,
	even when \eqref{N2cond} is not satisfied.

	\item \begin{equation*}
\left(\begin{array}{l}
   \frac{\beta-\mu-2a^{+}}{2a^{-}_{I}} + \frac{\sqrt{(\beta-\mu-2a^{+})^2(a^{-}_{O}-a^{-}_{I})^2 - 
   4a^{-}_{I}a^{+}(a^{-}_{O} - a^{-}_{I})(\beta-\mu-2a^{+})}}{2a^{-}_{I}(a^{-}_{O}-a^{-}_{I})}, \\
 \frac{\beta-\mu-2a^{+}}{2a^{-}_{I}} - \frac{\sqrt{(\beta-\mu-2a^{+})^2(a^{-}_{O}-a^{-}_{I})^2 - 
   4a^{-}_{I}a^{+}(a^{-}_{O} - a^{-}_{I})(\beta-\mu-2a^{+})}}{2a^{-}_{I}(a^{-}_{O}-a^{-}_{I})}
\end{array}
\right)
\end{equation*}
	\item \begin{equation*}
\left(\begin{array}{l}
   \frac{\beta-\mu-2a^{+}}{2a^{-}_{I}} - \frac{\sqrt{(\beta-\mu-2a^{+})^2(a^{-}_{O}-a^{-}_{I})^2 - 4a^{-}_{I}a^{+}(a^{-}_{O} - 
   a^{-}_{I})(\beta-\mu-2a^{+})}}{2a^{-}_{I}(a^{-}_{O}-a^{-}_{I})}, \\
 \frac{\beta-\mu-2a^{+}}{2a^{-}_{I}} + \frac{\sqrt{(\beta-\mu-2a^{+})^2(a^{-}_{O}-a^{-}_{I})^2 - 4a^{-}_{I}a^{+}(a^{-}_{O} - 
 a^{-}_{I})(\beta-\mu-2a^{+})}}{2a^{-}_{I}(a^{-}_{O}-a^{-}_{I})}
\end{array}
\right)
\end{equation*}
\end{enumerate}
\begin{proposition}
In the event that all four equilibria exist, the third and fourth equilibria are 
stable while the second one is a saddle point and is not stable.
\end{proposition}
\begin{proof} 
For the stability of the third and fourth, a computation shows that the eigenvalues of the Jacobian matrix of $\mathbf{F}(\mathbf{z}) = \left(F_1(\mathbf{z}), F_2(\mathbf{z})\right)$ with $F_1$ and $F_2$ as in \eqref{dfs1} at an equilibrium point $\mathbf{z^*} = (z_1^*,z_2^*)$, 
$$
J\left(\mathbf{z^*}\right) = \begin{pmatrix}
\beta - \mu - a^+ -2a^-_Iz_1^*-a^-_Oz^*_2 & a^+ - a^-_Oz^*_1\\
a^+ - a^-_Oz^*_2 & \beta - \mu - a^+ -2a^-_Iz_2^*-a^-_Oz^*_1\\
\end{pmatrix}
$$   are of the form
$$
\lambda_1 = \frac{A +\sqrt{B}}{2a^-_I(a^-_O-a^-_I)}\quad\text{and}\quad \lambda_2= \frac{A - \sqrt{B}}{2a^-_I(a^-_O-a^-_I)}
$$
for the third and fourth equilibrium points, where
\begin{align*}
A=(a^-_O-a^-_I)((&\mu-\beta)a^-_O+2a^+(a^-_O+a^-_I)),\\
B=\left(a^-_I - a^-_O\right)^2&\left[\beta^2 \left(-2a^-_I + a^-_O\right)^2 + \left(a^-_O\right)^2 (2 a^+ + \mu)^2\right. \\
&+ 4 \left(a^-_I\right)^2 \left(-3 \left(a^+\right)^2 + \mu^2\right)- 4 a^-_I a^-_O \left(2 \left(a^+\right)^2 + 3 a^+ \mu + \mu^2\right)\\ 
     &\left.- 2 \beta \left(4\mu \left(a^-_I\right)^2  + \left(a^-_O\right)^2 (2 a^+ + \mu) - 2 a^-_I a^-_O (3 a^+ + 2 \mu)\right)\right].
\end{align*}
It follows that $A < 0$ since the first factor of $A$ is positive and the second factor of $A$ is negative by \eqref{N2cond}. Since $A<0$, $B \leq 0$ implies that the real part of each eigenvalue, $\Re(\lambda_i) < 0$, $i=1,2$, and, therefore, the claimed stability. If $B> 0$, then consider\begin{align*}
A^2&-B=-4a^-_I\left(a^-_I-a^-_O\right)^2\left(\beta-2a^+-\mu\right)\left[\left(\beta-\mu\right)\left( a^-_I- a^-_O\right)+2a^+\left(a^-_O + a^-_I\right)\right]\end{align*} 
By \eqref{N2cond}, $\beta-2a^+-\mu > 0$, since $$\frac{a^-_O + a^-_I}{a^-_O - a^-_I}>1$$ and also by \eqref{N2cond},
$$\left(\beta-\mu\right)\left( a^-_I- a^-_O\right)+2a^+\left(a^-_O + a^-_I\right) < 0,$$ 
we conclude that $\Re(\lambda_i) <0$, $i=1,2$, in this case as well. 

To see that the second equilibrium point is not stable in this case, one can similarly evaluate the eigenvalues of the Jacobian matrix. A proof 
for general $N$ is given below, thus we omit the details here. 
\end{proof}

However, if \eqref{N2cond} is not satisfied, then we have 
only one non-trivial singular point, 
$$\left(\frac{\beta-\mu}{a^{-}_{I}+a^{-}_{O}},
\frac{\beta-\mu}{a^{-}_{I}+a^{-}_{O}}\right),$$ 
which is a stable equilibrium in this case.  
Note that this is the only non-trivial equilibrium 
if $a^{-}_{O}=0$, i.e., there is no suppression 
across boxes or classes.  This is the same 
equilibrium point, then, that is found for single 
boxes in the absence of any migration or mobility.

Note, also, that even if $a^{-}_{O}>a^{-}_{I}$ the 
existence of the third and fourth 
equilibria depends on low rates of migration between 
boxes (or social mobility between classes); 
these equilibria vanish if $a^{+}$ is too great.  
This is somewhat contrary to what one 
might suppose, that low rates of migration or mobility 
would keep the equilibria inside boxes at or near the original equilibria.

For three boxes or classes, $N = 3$, the results are similar.  
In particular, two equilibria always exist:
\begin{enumerate}[1)]
	\item The trivial singular point, an unstable equilibrium for $\beta>\mu$, at $(0,0,0)$, and

	\item $\left(\frac{\beta-\mu}{a^{-}_{I}+ 2a^{-}_{O}},
	\frac{\beta-\mu}{a^{-}_{I}+2a^{-}_{O}},\frac{\beta-\mu}{a^{-}_{I}+2a^{-}_{O}}\right)$.
 \end{enumerate}
If population suppression across boxes or classes does not occur, $a^{-}_{O} = 0$, the second of 
these is the only non-trivial equilibrium.  Otherwise, under additional conditions, including again, 
sufficiently low migration between boxes, multiple equilibria can exist.

  \begin{proposition}\label{eig}
 For $N\geq 2$, the points $\mathbf{0}$ and $\mathbf{z}^*\in\mathbb{R}^N$ with $$z_i^* = \frac{\beta - \mu}{a^{-}_{I}+(N-1)a^{-}_{O}}$$ are equilibrium points of \eqref{dfs2}, with $\mathbf{z}^*$ being stable only when
 \begin{equation}\label{condN}
 \left(\beta-\mu\right)\left(a^{-}_{O} - a^{-}_{I}\right) < 
 Na^{+}\left(a^{-}_{I} + (N-1)a^{-}_{O}\right)
 \end{equation}
 \end{proposition}
 
 \begin{proof}
One can check that $\mathbf{0}$ and $\mathbf{z}^*$  are equilibrium points by plugging them directly into \eqref{dfs3}. 
To see that $\mathbf{z}^*$ is stable under the condition \eqref{condN}, we again consider the Jacobian of $\mathbf{F}(\mathbf{z}) = \left(F_1(\mathbf{z}), \ldots, F_N(\mathbf{z})\right)$ with $F_i$ as in \eqref{dfs1}
with entries given by 
$$ 
\left(J\left(\mathbf{z^*}\right)\right)_{ij} = 
\begin{cases}
\beta - \mu - (N-1)a^{+} - 2a^{-}_{I}z^*_i - a^{-}_{O}(N-1)z^*_i, & i=j,\\
a^{+} - a^{-}_{O}z^*_i, & i\neq j
\end{cases}
$$
Given the special form of this matrix, the distinct eigenvalues are 
$$
\lambda_1 = \frac{(\beta-\mu)\left(a^{-}_{O}-a^{-}_{I}\right) - Na^{+}\left(a^{-}_{I}+(N-1)a^{-}_{O}\right)}{a^-_I + (N-1)a^-_O}\quad\text{and}\quad \lambda_2= \mu - \beta.$$
To see this, note that 
$$\left(J\left(\mathbf{z^*}\right) - \lambda_1 I_N\right)_{ij} 
= a^{+} - \frac{a^{-}_{O}(\beta - \mu)}{a^{-}_{I}+(N-1)a^{-}_{O}},$$ 
where $I_N$ is the $N\times N$ identity matrix, for all $i,j = 1,\ldots, N$. This matrix has rank $1$, thus the eigenspace of $\lambda_1$ 
is $(N-1)$-dimensional and so the multiplicity of $\lambda_1$ is $N-1$. To check that $\lambda_2$ is an eigenvalue 
with multiplicity $1$, we note that 
\begin{equation*}
\left(J\left(\mathbf{z^*}\right) - \lambda_2 I_N\right)_{ij} 
= 
\begin{cases}
 - (N-1)\left[a^{+} - \frac{a^{-}_{O}(\beta - \mu)}{a^{-}_{I} + (N-1)a^{-}_{O}}\right], & i = j,\\
a^{+} - \frac{a^{-}_{O}(\beta - \mu)}{a^{-}_{I} + (N-1)a^{-}_{O}}& i \neq j
\end{cases}
\end{equation*}
If we add each of rows $2$ through $N$ to the first row of $J\left(\mathbf{z^*}\right) - \lambda_2 I_N$, we obtain a zero row and it follows that $$J\left(\mathbf{z^*}\right) - \lambda_2 I_N$$ has rank $N-1$. Thus $\lambda_2$ is an eigenvalue of $J\left(\mathbf{z^*}\right)$ of multiplicity $1$.

$\lambda_1 < 0$ precisely when condition \eqref{condN} is satisfied and $\lambda_2<0$ from our assumption that $\beta > \mu$.
 \end{proof}


\subsection{Global limit theorems for $N$ boxes}
Here, we state a functional law of large numbers and functional central limit theorem, following \cite{tk70, tk71}.  
We now allow $L$ to vary, so we relabel slightly, setting $$z_{Li}(t) := \frac{n_i(t)}{L}, 
\qquad i=1,\ldots,N$$ and $Z_L(t) = (z_{L1}(t),\ldots,z_{LN}(t))$.
\begin{theorem}[Functional LLN]
Let $(z_1^*,\ldots, z_N^*)$ denote a unique stable equilibrium for the system given in \eqref{dfs1} 
and \eqref{dfs2}.  As $L \to \infty$, $$Z_L(t) \to Z(t)=(z_1(t),\ldots,z_N(t))$$ uniformly in 
probability, where $Z(t)$ is a deterministic process, the solution of
\begin{align}
	\frac{d z_j(t)}{dt} &= F_j(z_1(t),\ldots,z_N(t)),\ \ j = 1,\ldots, N, \label{cm1} \\
	z_1(0) &= z_1^*, \ldots,z_N(0) = z_N^*.\notag
	\end{align}
with $F_1, \ldots, F_N$ given in \eqref{dfs1}.
\end{theorem}

Next, define $g_{ij}(z_1,...,z_N)$:
\begin{align}
   g_{ii}(z_1,...,z_N) :&= \displaystyle\sum_{k_i=-1}^1 k_i^2 f(z_1,...,z_N,\cdot,k_i,\cdot)\notag\\
    &= \beta z_i + \mu z_i + a^{-}_{ii} z_i^2 + \displaystyle\sum_{j\ne i}(a^{-}_{ij} z_i z_j + a^{+}_{ij} z_i + a^{+}_{ji} z_j) \label{cov1} \\
   g_{ij}(z_1,...,z_N) &= g_{ji}(z_1,...,z_N) := \displaystyle\sum_{k_i,k_j=-1}^1 k_ik_j f(z_1,...,z_N,\cdot,k_i,\cdot,k_j,\cdot) \notag \\
  &= - a^{+}_{ij} z_i - a^{+}_{ji} z_j \ \ \ \ \mathrm{for \, } i\ne j \notag
\end{align}

\begin{theorem}[Functional CLT]
Let $z^*=(z_1^*, ...,z_N^*)$ denote a unique stable equilibrium for the system given in \eqref{dfs1} and \eqref{dfs2}.  If $\sqrt{L} \left(Z_L(0)-z^* \right) = \zeta_0$, the processes 
\[\zeta_L(t):=\sqrt{L}(Z_L(t) - z^*) \]
converge weakly in the space of cadlag functions on any finite time interval $[0, T]$ to an Ornstein-Uhlenbeck process (OUP) $\zeta(t)$ with initial value $\zeta_0$, infinitesimal drift given by 
\begin{align*}
q_1 := \frac{\partial F_1\left(z^*_1,...,z^*_N\right)}{\partial z_1},\ \ldots,\ 
q_N := \frac{\partial F_N\left(z^*_1,...,z^*_N\right)}{\partial z_N}
\end{align*}
and the infinitesimal covariance matrix with entries given by $$a_{ij} := g_{ij}\left(z^*_1,...,z^*_N\right).$$
\end{theorem}
Thus, for the single, symmetric positive equilibrium for $N=2$, with a single inner competition rate $a_I^-$, a single outer competition rate $a_O^-$, and a single migration rate $a^+$, the infinitesimal drift is: $$q_1=q_2=\frac{ -a^{-}_{I}(\beta-\mu)}{ a^{-}_{I}+ a^{-}_{O}}- a^{+}, $$
and the infinitesimal covariance matrix entries are:
\begin{align*}
a_{11} = a_{22} = \frac{2(\beta-\mu)(\beta+ a^{+})}{ a^{-}_{I}+ a^{-}_{O}},\quad
a_{12} = a_{21} = \frac{-2a^{+}(\beta-\mu)}{ a^{-}_{I}+ a^{-}_{O}}. 
\end{align*}


\section{Ergodicity for $N$ boxes}\label{sec:Ergodicity}
Assume there is no suppression of population across boxes, i.e., $a^-_{ij} = 0$ for $i \ne j$.
We also assume that $a^-_{ii} > 0$ for some $i = 1,\ldots,N.$
For $N$ boxes, let $\{X_n\}_{n=0}^{\infty}$ on $(\mathbb{Z_+})^N$ be the embedded discrete time random walk associated 
with the continuous random walk \eqref{RW}. For $\mathbf{x} = (x_1,\ldots,x_N)\in\mathbb{Z_+}^N,$ set 
$$c(\mathbf{x}) = \sum_{i=1}^N \left(\beta_i + \mu_i + \frac{a^-_{ii}}{L}x_i\right)x_i + \sum_{i,j=1, i\ne j}^N a_{ij}^+ x_i.$$
$\{X_n\}$ has transition probabilities, 
for $\mathbf{x}, \mathbf{y} \in (\mathbb{Z_+})^N$, $\mathbf{x}\neq\mathbf{0}$
\begin{equation}\label{Transition}
P(\mathbf{x}, \mathbf{y}) = \frac{1}{c(\mathbf{x})}\cdot
\begin{dcases}
 \,  \beta_i x_i & \text{if } \mathbf{y} = \mathbf{x} + e_i, i=1,\ldots,N \\
 \, \mu_i x_i + \frac{a^-_{ii}}{L} x_i^2 & \text{if } \mathbf{y} = \mathbf{x} - e_i, i=1,\ldots,N\\
 \,  a_{ij}^+ x_i & \text{if } \mathbf{y} = \mathbf{x} - e_i + e_j, i \ne j \\
 \, 0 & \text{ otherwise }
\end{dcases}
\end{equation}
and for $\mathbf{x} = \mathbf{0}$,
\begin{equation}\label{Transition0}
P(\mathbf{x}, \mathbf{y}) =
\begin{dcases}
 \, \frac{1}{N}  & \text{if } \mathbf{y} = \mathbf{0} + e_i, i=1,\ldots,N \\
 \, 0 & \text{ otherwise }
\end{dcases}
\end{equation}
Recall that we use $e_i\in\mathbb{Z}^N$ to denote the vector with $1$ in the $i^{th}$ position and $0$ 
everywhere else, and $\mathbf{0} = (0,\ldots,0)$. We impose here a reflective 
barrier at $\mathbf{0}$ with \eqref{Transition0}.

\begin{theorem}
A random walk with transition probabilities \eqref{Transition} 
and \eqref{Transition0}
is geometrically ergodic.  That is, it is positive recurrent 
with exponential convergence to a stable distribution.
\end{theorem}
\begin{proof}
Using Foster's \cite{Fo53} criterion,~\cite[Theorem~15.01]{MT93} 
(see also similar results in \cite{FMM95}) states 
that if there is a function $V:(\mathbb{Z_+})^N \to \mathbb{R}$ with $V(\mathbf{x}) \geq 1$ 
for all $\mathbf{x}\in(\mathbb{Z_+})^N$ such that, for a bounded set $B\subset (\mathbb{Z^+})^N$, 
constant $\lambda < 1$, and constant $b < \infty$,
\begin{equation}\label{ly1}
\sum_{\mathbf{y}\in(\mathbb{Z_+})^N}P(\mathbf{x},\mathbf{y})V(\mathbf{y}) \leq \lambda V(\mathbf{x}) + b\mathds{1}_B(\mathbf{x}), 
\end{equation}
then the Markov chain with probability transition matrix $P$ is geometrically ergodic.  
Here, $\mathds{1}_B(\mathbf{x})$ is the indicator function of $B$.
Let $$V(\mathbf{x}) = \alpha^{||\mathbf{x}||_1}, $$
where we will choose appropriate $\alpha > 1$, and $||\mathbf{x} ||_1$
is the $L^1$ norm of $\mathbf{x}$.
Note that, for $\mathbf{x}\in(\mathbb{Z_+})^N$, 
$$||\mathbf{x}||_1 = \sum_{i=1}^{N}|x_i| = \sum_{i=1}^{N}x_i.$$
Then, for $\mathbf{x} \notin B$ and if $\lambda\alpha > 1$, criterion~\eqref{ly1} is equivalent to
\begin{equation}\label{ly2}
\left(\alpha - \lambda\right)\sum_{i=1}^{N} \beta_ix_i + \left(1-\lambda\right)\sum_{i=1}^N A_i^+x_i
\leq \left(\lambda - \frac{1}{\alpha}\right)\sum_{i=1}^{N}\left(\mu_i + \frac{a_{ii}^-}{L}x_i\right)x_i
\end{equation}
for some $\lambda < 1$, where 
$$A_i^+ := \sum_{j=1}^Na_{ij}^+$$
is the total migration rate out of box $i$.
Let $$C_1 = \max_i \beta_i,\quad C_2 = \max_i  A^+_i,\quad
C_3 = \min_i \left\{\frac{a^-_{ii}}{L}: a^-_{ii}>0\right\}.$$
Then, for $\mathbf{x}\in(\mathbb{Z}_+)^N$ with 
\begin{equation}\label{x2ineq}
||\mathbf{x}||_2 \geq 
\frac{\sqrt{N}\left(\alpha C_1 + C_2\right)}{C_3\left(\lambda - 1/\alpha\right)}, 
\end{equation}
where $$||\mathbf{x}||_2 = \left(\sum_{i=1}^{N}x_i^2\right)^{1/2},$$
\begin{align*}
 \left(\alpha - \lambda\right)\sum_{i=1}^{N} \beta_ix_i + \left(1-\lambda\right)\sum_{i=1}^N A_i^+x_i & \leq
 \left(\alpha - \lambda\right)C_1 ||\mathbf{x}||_1 + \left(1-\lambda\right)C_2||\mathbf{x}||_1\\
 & \leq \sqrt{N}\left(\left(\alpha - \lambda\right)C_1  + \left(1-\lambda\right)C_2\right)||\mathbf{x}||_2\\
 & \leq \sqrt{N}\left(\alpha C_1  + C_2\right)||\mathbf{x}||_2\\
 & \leq C_3\left(\lambda - \frac{1}{\alpha}\right)||\mathbf{x}||_2^2\\
 & \leq \left(\lambda - \frac{1}{\alpha}\right)\sum_{i=1}^{N}\left(\mu_i + \frac{a_{ii}^-}{L}x_i\right)x_i,
\end{align*}
where the second inequaity is due to the Cauchy-Schwarz inequality, and the fourth inequality 
is due to our assumption \eqref{x2ineq}. The other inequalities follow from the definitions of 
$C_1, C_2,$ and $C_3$.

Thus, choose $$M = \frac{\sqrt{N}\left(\alpha C_1 + C_2\right)}{C_3\left(\lambda - 1/\alpha\right)},$$  
and let $$B = \left\{\,\mathbf{x}\in\left(\mathbb{Z}_+\right)^N\:\big|\: ||\mathbf{x}||_2 \leq M\,\right\}.$$  
Then $B$ is a bounded set, $V(\mathbf{x})\geq 1$ on $\left(\mathbb{Z}_+\right)^N$.
Let $$b = \max\left\{\left|\sum_{\mathbf{y}\in(\mathbb{Z}_+)^N}P(\mathbf{x},\mathbf{y})V(\mathbf{y}) - \lambda V(\mathbf{x})\right| :
\mathbf{x}\in\left(\mathbb{Z}_+\right)^N, ||\mathbf{x}||_2\leq M\right\}.$$
Then \eqref{ly1} is satisfied for all $\mathbf{x}\in(\mathbb{Z}_+)^N$.
\end{proof}

Suppose, finally, that we impose symmetric conditions on all of the boxes:
 \begin{enumerate}
 \item[1)] $\beta_i \equiv \beta$ and $\mu_i \equiv \mu$ for all $i$, with $\beta > \mu$,
 \item[2)] migration rates between all boxes are equal to $a^+$, 
 that is, $a^+_{ij} \equiv a^+$ for all $i, j$,
 \item[3)] suppression of population within its own box occurs 
 at the same rate for all boxes, i.e., $a^-_{ii} \equiv a_I^-$ for all $i$.
\end{enumerate}
Then, as is directly checked, the random walk has at least one non-trivial equilibrium point,  
that is, the drift vector $$\triangle \mathbf{x}:= \displaystyle\sum_{\mathbf{y}} P(\mathbf{x}, \mathbf{y}) \mathbf{y} - \mathbf{x} = 0$$ (cf. \cite{MT93})
at two points, the trivial point 0, and $\mathbf{x}$, where $$\frac{x_i}{L} = \frac{\beta-\mu}{a_I^-}$$ 
for all components $i$.  This follows from a computation for each component $i$ that 
	\[(\triangle \mathbf{x})_i = \frac{1}{c(\mathbf{x})} \left[ (\beta - \mu) x_i - \frac{a_I^-}{L} x_i^2 + a^+ \big(\sum_{j\ne i} x_j - (N-1) x_i \big)\right]. \]
The equilibrium result agrees with our earlier results in Proposition \ref{eig}.


\bibliographystyle{abbrvnat}
\bibliography{bibdata}
\end{document}